\documentclass{article}

\usepackage{pstricks}
\usepackage{pst-node}
\usepackage{pst-knot}
\usepackage{amsfonts}
\usepackage[margin=1in]{geometry}
\usepackage{graphicx}

\usepackage{amssymb}

\usepackage{url} 

\usepackage{amsmath}
\usepackage{amssymb}
\usepackage{amsthm}
\theoremstyle{plain}
\usepackage{tikz}
\usepackage{wasysym}
\usetikzlibrary{tqft}
\usetikzlibrary{positioning}
\usetikzlibrary{arrows.meta}

\usepackage{enumitem}
\usepackage{graphicx}
\usepackage{subcaption}
\usetikzlibrary{cd}
\allowdisplaybreaks
\usetikzlibrary{decorations.pathmorphing}
\usetikzlibrary{calc}

\usepackage[numbers]{natbib}

\usepackage{mathtools}
\usepackage{stmaryrd}

\usepackage{hyperref}

\usepackage{mathtools}

\usepackage[margin=1cm, font=small,labelfont=bf]{caption}
\theoremstyle{definition}
\newtheorem{theorem}{Theorem}[section]
\newtheorem*{theorem*}{Theorem \ref{KirkMain}}
\newtheorem{ex}[theorem]{Example}
\newtheorem{defn}[theorem]{Definition}
\newtheorem{notn}[theorem]{Notation}
\newtheorem{rmk}[theorem]{Remark}
\newtheorem{conv}[theorem]{Convention}
\newtheorem{lemma}[theorem]{Lemma}

\newtheoremstyle{case}{}{}{}{}{}{:}{ }{}
\theoremstyle{case}
\newtheorem{case}{Case}

\makeatletter\@addtoreset{case}{theorem}\@addtoreset{case}{lemma}\makeatother

\newcommand{\wtilde}[1]{{\stackrel{\sim}{\smash{#1}\rule{0pt}{1.1ex}}}} 
\newcommand{\Cone}{\textnormal{Cone}} 
\newcommand{\Hom}{\textnormal{Hom}} 
\newcommand{\Decat}{\textnormal{Decat}} 
\newcommand{\Cob}{\mathbf{Cob}} 
\newcommand{\rank}{\textnormal{rank}} 
\newcommand{\grab}{\mathbf{grAb}}
\newcommand{\nTQFTR}{n\mathbf{TQFT}_R}
\newcommand{\TQFTR}{2\mathbf{TQFT}_R}
\newcommand{\cFAR}{\mathbf{cFA}_R}
\newcommand{\kvect}{\Bbbk\textnormal{-}\mathbf{Vect}} 
\newcommand{\Rmod}{R\textnormal{-}\mathbf{mod}}

\newcommand{\Zndiag}{\Z^n_+\textnormal{-}\mathbf{diag}} 
\newcommand{\Ob}{\textnormal{Ob}} 
\newcommand{\Mor}{\textnormal{Mor}} 
\newcommand{\N}{\mathbb{N}} 
\newcommand{\R}{\mathbb{R}} 
\newcommand{\Z}{\mathbb{Z}} 
\newcommand{\Q}{\mathbb{Q}} 
 
\newcommand{\vcplx}{\mathcal{C}^*(D,\vec{x})} 
 
\newcommand{\vhcplx}{\mathcal{H}^*(D,\vec{x})} 
\newcommand{\vkhcplx}{\mathcal{H}^k(D,\vec{x})} 
\newcommand{\tvcplx}{\mathcal{C}^*(T_{2,n},\vec{x})} 
\newcommand{\tvkcplx}{\mathcal{C}^k(T_{2,n},\vec{x})} 
\newcommand{\tvhcplx}{\mathcal{H}^*(T_{2,n},\vec{x})} 
\newcommand{\tvkhcplx}{\mathcal{H}^k(T_{2,n},\vec{x})} 
\newcommand{\vsdiag}{[\![D]\!]}
\newcommand{\inv}{\text{inv}}

\definecolor{mygreen}{RGB}{34,139,34}
\input{NewColorSaveBox}
\input{CleanSaveBox}
\input{relevant_smoothings}

\title{A Categorification of the Vandermonde Determinant}

\title{A Categorification of the Vandermonde Determinant}
\author{Alex Chandler}

\begin{document}
\maketitle
\begin{abstract}
In the spirit of Bar Natan's construction of Khovanov homology, we give a categorification of the Vandermonde determinant. Given a sequence of positive integers $\vec{x}=(x_1,...,x_n)$, we construct a commutative diagram in the shape of the Bruhat order on $S_n$ whose nodes are colored smoothings of the $2$-strand torus link $T_{2,n}$, and whose arrows are colored cobordisms. An application of a TQFT to this diagram yields a chain complex whose Euler characteristic is the Vandermonde determinant evaluated at $\vec{x}$. A generalization to arbitrary link diagrams is given, producing categorifications of certain generalized Vandermonde determinants. We also address functoriality of this construction. 



\end{abstract}
\tableofcontents
\section{Introduction}

Categorification, as envisioned by Crane and Frenkel in \cite{crane1994four}, 
can be thought of as the process of interpreting a set theoretic or algebraic structure as a `shadow' of a category theoretic analogue. A categorification can endow familiar mathematical objects with richer structure, and can shed light on the structure by providing new tools to study it which were unavailable in the original setting. Categorification can be thought of as something like a mathematician's version of Plato's allegory of the cave. In Plato's allegory, a group of prisoners are confined in a cave facing a blank wall their entire lives. Outside the cave, there is a fire which casts shadows on the cave wall as objects pass in front of it. The prisoners have no concept of what the objects are and are aware only of the motions of the shadows on the wall. 
Plato's idea of a philosopher is one who escapes the cave and learns the true nature of these objects. Inspired by the ideas of Crane and Frenkel, we now attempt to ``escape the cave".

Linear algebra has proven to be an indispensable tool, having influence throughout mathematics and science. Therefore, one might expect that categorifying concepts in linear algebra would be of comparable importance. For instance, 
Elias and Hogancamp categorify the concepts of eigenvalues, eigenvectors, and diagonalization in \cite{elias2017categorical}. 
Just as diagonalization has various uses in representation theory, categorical diagonalization has proven useful in categorical representation theory. In \cite{elias2017categorical}, Elias and Hogancamp diagonalize the full twist Rouquier complex, and as an application are able to categorify the Young idempotents. The trace of the Coxeter matrix has been categorified by Happel in \cite{happel1997trace} using Hochschild homology. Happel's approach provides a topological interpretation of the condition that certain algebras have Coxeter matrix with trace $-1,$ e.g., the trace of the Coxeter matrix of the path algebra of a finite quiver without oriented cycles is -1 if and only if the underlying graph is a tree. Stolz and Teichner in \cite{stolz2012traces} provide a more general definition of a trace for monoidal categories, yielding applications to the partition function in super symmetric field theories.

This paper is motivated by M. Khovanov's encouragement to categorify special types of determinants, as a step towards categorifying other concepts in linear algebra.
We consider the {\it Vandermonde determinant}, usually defined as $\wtilde{V}(\vec{x})=\det (x_i^{j-1})_{i,j=1}^n$ where $\vec{x}=(x_1,\dots,x_n)$ is a list of variables. For purposes of categorification, we find it more convenient to consider the following rescaling:
\begin{equation}
V(\vec{x})=\begin{vmatrix}
  x_1 & x_1^2 & \cdots& x_1^{n} \\
  x_2 & x_2^2 & \cdots& x_2^{n} \\
\vdots & \vdots & \ddots &  \vdots \\
x_n & x_n^2 & \cdots &  x_n^{n} 
\end{vmatrix}
=\det (x_i^j)_{i,j=1}^n
\end{equation}
and require that $\vec{x}=(x_1,\dots,x_n)\in\Z_+^n$, where $\Z_+$ denotes the positive integers. Linearity of the determinant yields the relation $V(\vec{x})=x_1\dots x_n\wtilde{V}(\vec{x})$. The Vandermonde determinant and its properties are useful in several areas of mathematics. The nonvanishing of $\wtilde{V}(\vec{x})$ for distinct values of its inputs shows that the polynomial interpolation problem is uniquely solvable. It is a standard result that $\wtilde{V}(\vec{x})=\prod_{i<j}(x_j-x_i)$, and thus any alternating polynomial in variables $x_1,\dots x_n$ is divisible by $\wtilde{V}(\vec{x})$. For any partition $\lambda=(\lambda_1,\dots,\lambda_n)$ one can define the generalized Vandermonde determinant $\wtilde{V}_{\lambda}(\vec{x})=\det (x_i^{j+\lambda_{n-j+1}})_{i,j=1}^n$. The quotient $\wtilde{V}_\lambda(\vec{x})/\wtilde{V}(\vec{x})$ arises in the Frobenius character formula which can be used to compute characters of representations of the symmetric group. The Vandermonde determinant also appears in the theory of BCH code, Reed-Solomon error correction codes \cite{klove1999algebraic}, and can be used to define the discrete Fourier transform \cite{massey1998discrete}.

The following expressions for the Vandermonde determinant will prove to be useful for the purposes of categorfication: 
\begin{align}
V(\vec{x})&=\sum_{\pi\in S_n}(-1)^{\text{sgn}(\pi)} x_1^{\pi(1)}x_2^{\pi(2)}... \ x_n^{\pi(n)} \label{eq1} \\
&= \sum_{k\geq 0}(-1)^{k} \ \bigg[\sum_{\substack{\pi\in S_n \\ \text{inv}(\pi)=k}} x_1^{\pi(1)}x_2^{\pi(2)}... \ x_n^{\pi(n)}\bigg]\label{eq2}
\end{align}
where $S_n$ is the symmetric group on $n$ letters. With the above expression in mind, and comparing with Example \ref{chaincomplexex}, it makes sense to view the Vandermonde determinant as the Euler characteristic of some (co)homology theory. The goal of this paper is to categorify the Vandermonde determinant. To this end, we use some ideas from combinatorics and topology to construct a cohomology theory whose Euler characteristic is equal to the Vandermonde determinant. We accomplish this in a fashion similar to Khovanov's categorification of the Jones polynomial. The method used here is most strongly inspired by Bar-Natan's description \cite{bar2002khovanov,bar2005khovanov} of Khovanov's categorification. 

Specifically, given a link diagram $D$ with an ordering $c_1,\dots,c_n$ of its crossings, and a list $\vec{x}=(x_1,\dots,x_n)$ of natural numbers,  we construct a cochain complex $\vcplx$. Let $\vhcplx$ denote the cohomology of $\vcplx$. The main result of this paper, Theorem \ref{KirkMain}, concerns the behavior of $\vhcplx$ for a specific choice of diagram $D$. Let $B_2$ denote the braid group on 2 strands generated by the positive crossing $\sigma_1=\usebox\inlinecrossing \ \in B_2$. Let $T_{2,n}$ denote the diagram of the $(2,n)$-torus link obtained as the braid closure of $\sigma_1^n\in B_2$ (see Figure \ref{fig:torusa}).

\begin{theorem*} The Euler characteristic of $\mathcal{H}^*(T_{2,n},\vec{x})$ is equal to  the Vandermonde determinant $V(\vec{x})$:
$$\sum_{k\geq 0}(-1)^k\dim\tvkhcplx=V(\vec{x}).$$
Furthermore, the complex $\tvcplx$ does not depend on the ordering of the crossings of $T_{2,n}$. 
\end{theorem*}

The rest of the paper is organized as follows. Section \ref{categbackground} provides the basic tools and examples of categorification needed. Section \ref{sec2} is a review of symmetric group notations and the Bruhat order. Section \ref{sec3} gives relevant definitions and results about TQFTs and Frobenius algebras. Section \ref{sec4} introduces colored TQFTs: a generalization of standard TQFTs needed in this paper to accommodate the Vandermonde determinant's $n$-variables. In Section \ref{sec4} we construct, for each link diagram, a functor from the Bruhat order to a category of colored cobordisms. In Section \ref{Sectionthree} we construct the complex $\vcplx$, and show that its cohomology $\vhcplx$ categorifies certain generalized Vandermonde determinants. Choosing $D=T_{2,n}$ produces a complex whose cohomology $\tvhcplx$ categorifies the Vandermonde determinant $V(\vec{x})$. In Section \ref{functsection} we give a sense in which this categorification is functorial. In Section \ref{categarbit} we categorify determinants of positive integer valued matrices. Finally, Section \ref{lastsection} is dedicated to some remaining questions and possible applications related to the categorifications presented in this paper.

\textbf{Acknowledgements}: The author would like to thank  Mikhail Khovanov and Radmila Sazdanovi\'{c} for the encouragement and initial motivation for this project, Tye Lidman and Radmila Sazdanovic for thoughtful comments, corrections, and suggestions, and Michael Breen-McKay for several fruitful conversations.

\section{Background}

\subsection{Categorification}
\label{categbackground}

Categorification is the process of finding category theoretic analogues of set theoretic or algebraic structures.  Following Baez and Dolan in \cite{baez1998categorification} we give an analogy between set theory and category theory in Figure \ref{analogy}.

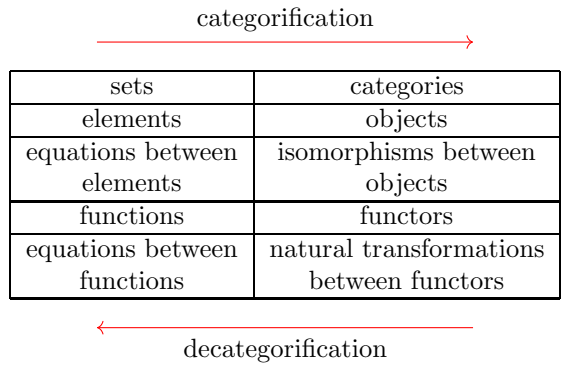
\begin{figure}[h]
\centering
\begin{tikzpicture}
\draw [->,red] (-2.5,1.9)--(2.5,1.9);
\node at (0,2.2){categorification};
\node at (0,0){\usebox\cattable};
\draw [<-,red] (-2.5,-1.9)--(2.5,-1.9);
\node at (0,-2.2){decategorification};
\end{tikzpicture}
\caption{An analogy between set theory and category theory.}
\label{analogy}
\end{figure}

To categorify a set $S$ one should find a category $\mathcal{C}$ and a function $\Decat:\{\textnormal{isom. classes of} \ \mathcal{C}\}\to S$. If $S$ has some extra structure (for example a group or ring structure), then $\mathcal{C}$ should have natural isomorphisms between objects (known as \textit{coherence conditions}) which descend under the function $\Decat$ to the appropriate structural equations between elements in $S$. In this context, we say that $\mathcal{C}$ \textit{categorifies} $S$, the object $x\in\Ob(\mathcal{C})$ \textit{categorifies} $\Decat([x])\in S$, and $\Decat([x])$ is the \textit{decategorification} of $x$. In general, there may be many ways to categorify a given object, but some ways are more useful/natural than others. The ``right" categorification should be one that not only lifts structures present at the decategorified level, but also introduces interesting new structures.
Next we provide some standard examples of categorification. Examples \ref{categnatural} and \ref{chaincomplexex} can be thought of as the basic building blocks for other interesting  categorifications such as Examples \ref{categeuler} and \ref{categkhovanov}. 

 \begin{ex}
The natural numbers $(\N,+,\cdot)$ form the structure of a \textit{rig}, that is, a ring without necessarily having additive inverses. The rig $(\N,+,\cdot)$ is categorified by the category $\kvect$ of finite dimensional vector spaces over $\Bbbk$ where $\Bbbk$ is any field. Decategorification is done by taking the dimension of the vector space $\Decat(V)=\dim(V)$. Direct sums and tensor products of vector spaces behave nicely in this regard: $\dim(V\oplus W)=\dim(V)+\dim(W)$ and $\dim(V\otimes W)=\dim(V)\text{dim}(W)$. $\kvect$ has all of the appropriate coherence conditions needed to categorify $\N$ as a rig, and for this reason is called a \textit{rig category}. Not only are the nice properties of the natural numbers lifted to the category of vector spaces, but we have a variety of new tools and structures available (that is, all of linear algebra).
\label{categnatural}
\end{ex}

\begin{ex} The ring $(\Z,+,\cdot)$ of integers is categorified by the category $\mathcal{C}(\kvect)$ of chain complexes of $\Bbbk$-vector spaces. Integers are categorified by chain complexes of vector spaces $V_*$ where decategorification is done by taking the Euler characteristic of a chain complex $\Decat(V_*)=\chi(V_*)=\sum_{n\in\mathbb{Z}}(-1)^n\text{dim}(V_i)$. Again, direct sums and tensor products behave nicely: $\chi(V_*\oplus W_*)=\chi(V_*)+\chi(W_*)$ and $\chi(V_*\otimes W_*)=\chi(V_*)\chi(W_*)$. $\mathcal{C}(\kvect)$ satisfies all of the appropriate coherence conditions needed to categorify $\Z$ as a ring. Furthermore, we have some very nice additional structure not available at the decategorified level (that is, all of homological algebra). In the same manner, we can think of integers as being categorified by cochain complexes $V^*$. Note that in this example and the previous one, $\kvect$ can be replaced with the category $\Rmod$ of finitely generated modules over a commutative ring $R$ (and $\dim$ should be replaced by $\rank$).
\label{chaincomplexex}
\end{ex}

In the previous two examples, we see categorifications of $\N$ as a rig and of $\Z$ as a ring. As of yet, there is no known categorification of $\Q$ or $\R$ as rings. As a step in this direction, Khovanov and Tian give a categorification of $\Z[\frac{1}{2}]$ in \cite{khovanov2017categorify}.
Next, we provide two examples of interesting and useful categorifications, both providing a rich variety of new tools and structures in their respective fields. Both of the following examples are considered to be the ``right'' categorifications. 

\begin{ex} Let $\mathbf{Sim}$ denote the category of  simplicial complexes and simplicial maps, and let $\Delta\in\Ob(\mathbf{Sim})$. Let  $f_n(\Delta)$ denote the number of $n$ dimensional faces in $\Delta$. The Euler characteristic $\chi(\Delta)=\sum_{n\geq 0}(-1)^nf_n(\Delta)$ is categorified by the simplicial homology $H_*(\Delta)$ of $\Delta$ in the sense that 
\begin{equation}
\sum_{n\geq 0}(-1)^n\dim H_n(\Delta)=\chi(\Delta).
\end{equation}
For this reason, given a chain complex $C_*$, the quantity $\sum_{n\in\Z}(-1)^n \dim C_n=\sum_{n\in\mathbb{Z}}(-1)^n \dim H_n$ is called the Euler characteristic of $C_*$. Simplicial homology contains the information of the Euler characteristic, but also has much more information about the space. For example, each homology group $H_n(\Delta)$ is a topological invariant, and the rank of the homology group $H_n(\Delta)$ indicates the number of `holes' of dimension $n$ in $\Delta$. 

The Euler characteristic of simplicial complexes is just a function $\chi:\Ob(\mathbf{Sim})\rightarrow\mathbb{Z}$ whereas homology is a functor $H_*:\mathbf{Sim}\rightarrow \grab$, where $\grab$ denotes the category of graded Abelian groups. Thus, not only do we get a stronger invariant, but for each simplicial map $\Delta\to\Delta^\prime$ we get an induced map $H_*(\Delta)\to H_*(\Delta^\prime)$. This functoriality is what gives simplicial homology its true power as compared to the Euler characteristic. For instance, functoriality of simplicial homology provides an easy proof of the Brouwer fixed-point theorem.
\label{categeuler}
\end{ex}
\begin{ex} A more recent example of categorification is given by the Khovanov homology. Khovanov homology is a bigraded Abelian group which categorifies the Jones polynomial of a link in the sense that one may recover the Jones polynomial by taking the (graded) Euler characteristic of the Khovanov homology. This theory was developed by Khovanov in \cite{khovanov1999categorification}, with more topological interpretations feature in the work of D. Bar-Natan in \cite{bar2005khovanov} and O. Viro in \cite{viro2004khovanov}. The Jones polynomial is a powerful link invariant, and the Khovanov homology is a strictly stronger link invariant. Moreover, Khovanov homology can be shown to be functorial, meaning that a cobordism between links induces maps between the Khovanov homologies of those links. In particular this functoriality provides numerical invariants of 2-knots by looking at the induced maps of cobordisms with no boundary. Functoriality in Khovanov homology has also proven useful in producing a lower bound for the slice genus of a knot. Rasmussen used this fact in \cite{rasmussen2010khovanov} to give a combinatorial proof of the Milnor conjecture, computing the slice genus of torus links.

The success of Khovanov's categorification of the Jones polynomial has inspired many other categorifications. One can find many other successful categorifications of link polynomials which are defined in a similar manner. For example, Ozsv\'{a}th and Szab\'{o}'s  categorification of the Alexander polynomial in \cite{ozsvath2004holomorphic} done independently by Rasmussen in \cite{rasmussen2003floer}. In fact, there is a whole family $\{P_n\}_{n\in\N}$ of polynomial link invariants (including the Jones polynomial $P_2$, and the Alexander polynomial $P_0$, as special cases) which are categorified by bigraded homology theories \cite{khovanov2004matrix}. Khovanov and Rozansky's categorification of the HOMFLY-PT polynomial in \cite{khovanov2008matrix} and \cite{khovanov2007triply} requires a triply-graded link homology theory. Similar in spirit, we also have Helme-Guizon and Rong's categorification of the chromatic polynomial in \cite{helme2005categorification}, Hepworth's categorification of the magnitude of a graph in \cite{hepworth2015categorifying}, and many more. The categorification given in this paper is of a similar nature.
\label{categkhovanov}
\end{ex}

\subsection{The Symmetric Group and its Bruhat Order}
\label{sec2}

For $\vec{x}\in\Z_+^n$, the Vandermonde determinant $V(\vec{x})$ can be expressed, via the equation (\ref{eq1}), as a sum over the symmetric group on $n$-letters, $S_n$. In this section, we recall some familiar notations, and a partial order on $S_n$ called the Bruhat order.

Let $[n]$ denote the set of numbers $\{1,2,...,n\}$ and let $S_n$ denote the group of bijections of $[n]$.  The \textit{one-line notation} for $\pi\in S_n$ is
$\pi=\pi_1\pi_2...\pi_n$ where $\pi_i=\pi(i)$.  Let $(a_1, a_2, ... , a_k)$ denote the \textit{cycle} sending $a_i$ to $a_{i+1}$ for $1\leq i\leq k-1$, sending $a_k$ to $a_1$, and fixing all elements of $[n]$ which do not appear in the list $a_1,...,a_k$. A \textit{transposition} is a cycle of the form $(a_1,a_2)$. Any $\pi\in S_n$ can be written as a product of transpositions $\pi=\tau_1\dots \tau_k$ and the quantity $\text{sgn}(\pi)=(-1)^k$ is well defined. An \textit{inversion} in $\pi$ is a pair $(i,j)\in[n]\times [n]$ with $i<j$ and $\pi(i)>\pi(j)$. Recall that $\text{sgn}(\pi)=(-1)^{\text{inv}(\pi)}$ where $\text{inv}(\pi)$ denotes the number of inversions in $\pi$.


%

The categorification of $V(\vec{x})$ presented in Section \ref{Sectionthree} uses the following partial order on $S_n$, which can be defined for any Coxeter group. Here we assume some basic knowledge of poset theory. If needed, the relevant definitions can be found in \cite{stanley1998enumerative}. Recall that a \textit{cover relation} in a poset $P$ is an order relation $x<y$ for which there exists no $z$ with $x<z<y$ and in this case we write $x\lessdot y$. If $P$ is a finite poset, then knowing the cover relations in $P$ is enough to know all order relations (in this case $x<y$ if and only if there is a chain of cover relations $x=x_0\lessdot x_1\lessdot\dots\lessdot x_k=y$).

\begin{defn}\label{bruhat} 
The (\textit{strong}) \textit{Bruhat order} on $S_n$ is the poset determined by the following cover relations: given $\pi, \sigma\in S_n$, define $\pi\lessdot \sigma$ if and only if $\sigma=\pi\tau$ where $\tau$ is a transposition and $\text{inv}(\sigma)=\text{inv}(\pi)+1$.
\end{defn}

The cover relations in the Bruhat order can be described conveniently by looking at $\pi$ and $\sigma$ in one-line notation and $\tau$ in cycle notation. We observe that multiplying on the right by the transposition $(i,j)$ swaps the $i^\text{th}$ and $j^\text{th}$ positions in one-line notation:
$\pi_1...\pi_i...\pi_j...\pi_n \cdot (i,j)=\pi_1...\pi_j...\pi_i...\pi_n.$
 Thus $\pi \lessdot \sigma$ in $S_n$ if and only if $\sigma$ can be produced by finding and interchanging two entries $\pi_i, \pi_j$ of $\pi$ in one-line notation with $i<j$ and $\pi_i<\pi_j$ such that none of the numbers $\pi_{i+1},...,\pi_{j-1}$ are in the $\Z$-interval $(\pi_i,\pi_j)$.
 
Recall, a poset $P$ is {\it ranked} if there is an order preserving function $\rho:P\to \Z$ (called the {\it rank function}) for which $x\lessdot y$ implies $\rho(y)=\rho(x)+1$. In a ranked poset, one can define the {\it length} of an interval $[x,y]$ to be $\rho(y)-\rho(x)$. For example, consider the poset $D=\{\hat{0},a,b,\hat{1}\}$ with cover relations $\hat{0}\lessdot a\lessdot \hat{1}$ and $\hat{0}\lessdot b\lessdot\hat{1}$ with $a$ and $b$ incomparable. $D$ has rank function $\rho(\hat{0})=0$, $\rho(a)=1=\rho(b)$, $\rho(\hat{1})=2$, and Hasse diagram $D=\usebox\diamondd$. Any poset isomorphic to $D$ is called a \textit{diamond}. If $P$ is a ranked poset such that every interval of length 2 is a diamond, then $P$ is said to be a \textit{thin poset}. The following well-known theorem (see for example Section 2.7 in \cite{bjorner2006combinatorics}) will be essential in our categorification of the Vandermonde determinant.
\begin{theorem} The Bruhat order on $S_n$ is a thin poset with rank function $\pi\mapsto\text{inv}(\pi)$.
\label{bruhatthin}
\end{theorem}
\begin{rmk} The definition of Khovanov homology depends strongly on the fact that the Jones polynomial can be expressed as a rank alternating sum over a thin poset (the Boolean lattice). The categorification of the Vandermonde determinant described in Section \ref{Sectionthree} mimics the construction of Khovanov homology and is motivated by the fact that the Vandermonde determinant is also a rank alternating sum over a thin poset (the Bruhat order on $S_n$). 
\end{rmk}
\subsection{Cobordisms, TQFTs, and Frobenius Algebras}
\label{sec3}
Bar-Natan's description \cite{bar2005khovanov} of Khovanov homology 
is gotten by applying a TQFT to
a certain diagram in the category $\Cob_2$ of 2-dimensional cobordisms in the shape of the Boolean lattice. Our categorification of $V(\vec{x})$ is defined similarly but will instead be gotten by applying a special colored TQFT (see Definition \ref{SCTQFTdef}) to a diagram in the category $\Cob_2^{\vec{x}}$ of 2-dimensional colored cobordisms (see Definition \ref{colcobcat}), in the shape of the Bruhat order on $S_n$. In this section we review cobordisms, TQFTs, and the equivalence between 2-dimensional TQFTs and commutative Frobenius algebras.
\begin{defn} \label{cobcateg} Let $\Cob_n$ denote the the following category: objects are closed oriented ${(n-1)}$-dimensional manifolds, morphisms between objects $M$ and $N$ are (orientation preserving) diffeomorphism classes of oriented $n$-dimensional manifolds $W$ with $\partial W=-M\amalg N$, where $-M$ denotes $M$ with the opposite orientation and $\amalg$ denotes disjoint union. In this case we write $\partial_i W=M$ and $\partial_o W=N$ and say that $M$ is the \textit{in-boundary} and $N$ is the \textit{out-boundary} of $W$. Morphisms in $\Cob_n$ are called {\it oriented cobordisms}. Given $A\in \Hom(M,N)$ and $B\in\Hom(N,L)$ define the composition $B\circ A$ by gluing $B$ to $A$ along the identity map on $N$. See \cite{kock2004frobenius} for details on the well-definedness of this construction. 
\end{defn}
Given two $(n-1)$-manifolds, $M$ and $N$, one can form the disjoint union $M\amalg N$ which is again an $(n-1)$ manifold. Similarly, one can form disjoint unions of cobordisms. Disjoint unions endow the category $\Cob_n$ with the structure of a monoidal category. Given a diffeomorphism $\phi:M\rightarrow N$, one can form a cobordism from $M$ to $N$ by first forming the cylinder $[0,1]\times M$ and gluing $N$ to $\{1\}\times M$ via the diffeomorphism $\phi$. Notice that $M\amalg N$ is diffeomorphic to $N\amalg M$ via the map $\tau_{M,N}$ which interchanges factors. The cobordism corresponding to $\tau_{M,N}$ from $M\amalg N$ to $N\amalg M$ is called a {\it twist cobordism}. Twist cobordisms act as a symmetric braiding, and thus endow $\Cob_n$ with the structure of a symmetric monoidal category. In the case that $M$ and $N$ are both circles, the corresponding twist cobordism will be denoted as shown in Figure \ref{fig:7e}. 

In this paper, we are concerned only with $\Cob_2$ where cobordisms are easy to visualize and classify. The objects of $\Cob_2$ are closed oriented 1-manifolds (disjoint unions of oriented circles). By the classification of surfaces, morphisms in $\Cob_2$ are oriented surfaces with oriented boundary. Thus in dimension 2, connected cobordisms are classified by three quantities: genus, number of in-boundary components, and number of out-boundary components. In our pictures, we will draw the in-boundary on the bottom of the cobordism and the out-boundary on the top (so pictures go from the  bottom up). A standard result of Morse theory is that any cobordism in dimension 2 can be built by gluing and taking disjoint unions of the basic building blocks shown in Figure \ref{buildingblocks}.
\begin{figure}[h]
\begin{subfigure}{0.18\textwidth}
\centering
\usebox\generatorcap
\caption{} \label{fig:7a}
\end{subfigure}
\begin{subfigure}{0.18\textwidth}
\centering
\usebox\generatorcup
\caption{} \label{fig:7b}
\end{subfigure}
\begin{subfigure}{0.18\textwidth}
\centering
\usebox\generatorpants
\caption{} \label{fig:7c}
\end{subfigure}
\begin{subfigure}{0.18\textwidth}
\centering
\usebox\generatorcopants
\caption{} \label{fig:7d}
\end{subfigure}
\begin{subfigure}{0.18\textwidth}
\centering
\usebox\generatortwist
\caption{} \label{fig:7e}
\end{subfigure}
\caption{(a) The \textit{cap} cobordism from $\Circle$ to $\varnothing$. (b) The \textit{cup} cobordism from $\varnothing$ to $\Circle$. (c) The \textit{pants} cobordism from $\Circle\Circle$ to $\Circle$. (d) The \textit{copants} cobordism from $\Circle$ to $\Circle\Circle$. (e) The twist cobordism from $\Circle\Circle$ to $\Circle\Circle$.} \label{buildingblocks}
\end{figure}
In Section \ref{mainresult}, we will categorify the Vandermonde determinant by first constructing a diagram in a category of cobordisms, and then using this diagram to construct a cochain complex which contains $V(\vec{x})$ as its Euler characteristic. To do this, we need a way to pass from the category of cobordisms to a category of modules.
\begin{defn} An $n$-dimensional \textit{topological quantum field theory (TQFT)} is a symmetric monoidal functor from $\Cob_n$ to $\Rmod$ for some ring $R$. TQFTs over a ring $R$ form a category, denoted $\nTQFTR$, where morphisms are monoidal natural transformations. 
\label{TQFTdef}
\end{defn}
%
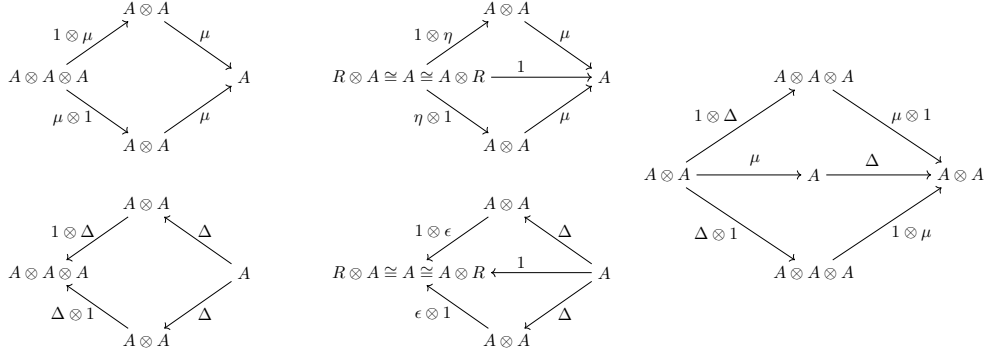
\begin{figure}[h]
\centering
\begin{tikzpicture}[scale=0.65, every node/.style={transform shape}]
\useasboundingbox (0,0.5) rectangle (19,7.5);
\node at (2,6){\usebox\algassoc};
\node at (9,6){\usebox\algunit};
\node at (2,2){\usebox\coalgassoc};
\node at (9,2){\usebox\coalgunit};
\node at (16,4){\usebox\frobeniusrel};
\end{tikzpicture}
\caption{Axioms of a Frobenius algebra. The diagram on the right illustrates the Frobenius relation. The other four diagrams illustrate the associativity, unit, coassociativity, and counit relations.}
\label{AxiomsFrobenius}
\end{figure}
\begin{defn} A \textit{Frobenius algebra} is a tuple $(A,\mu,\eta,\Delta,\varepsilon)$ such that $(A, \mu, \eta) $ is a unital associative algebra over a ring $R$ and $(A,\Delta, \varepsilon)$ is a counital coassociative coalgebra over $R$ for which the Frobenius relation holds: \[(\mu\otimes 1)\circ (1\otimes \Delta)=\Delta\circ\mu=(1\otimes\mu)\circ(\Delta\otimes 1).\]
Commutative Frobenius algebras over a ring $R$ form a category, denoted $\cFAR$, whose morphisms are homomorphisms of Frobenius algebras, that is, algebra homomorphisms which are also coalgebra homomorphisms.
\label{FrobAlgDef}
\end{defn}
\noindent These two definitions, \ref{TQFTdef} and \ref{FrobAlgDef}, look quite different at face value, however 2-dimensional TQFTs and (commutative) Frobenius algebras actually encode the same information. 
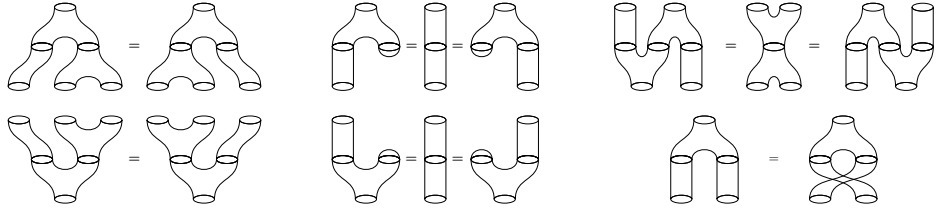
\begin{figure}[h]
\centering
\begin{tikzpicture}[scale=0.5, every node/.style={transform shape}]
\useasboundingbox (-2,1) rectangle (22,6);
\node at (1,5){\usebox\assoc};
\node at (9,5){\usebox\multunit};
\node at (1,2){\usebox\coassoc};
\node at (9,2){\usebox\comultcounit};
\node at (18,5){\usebox\frobrel};
\node at (18,2){\usebox\frocomm};
\end{tikzpicture}
\caption{Equalities of cobordisms which correspond to axioms of Frobenius algebras (compare to Figure \ref{AxiomsFrobenius}). The equality at the top right corresponds to the Frobenius relation, and the equality at the bottom right guarantees that the corresponding Frobenius algebra is commutative. The other four equalities correspond to the associativity, unit, coassociativity, and counit relations.}
\label{CobordismsFrobenius}
\end{figure}
Upon comparison of the relations in $\Cob_2$ shown in Figure \ref{CobordismsFrobenius} and the axioms of commutative Frobenius algebras in Figure \ref{AxiomsFrobenius}, we see that we can go back and forth between TQFTs and commutative Frobenius algebras via 
\begin{equation}
F\longleftrightarrow\bigg(
F\big(\Circle\big), 
F\big(\raisebox{-1mm}{\usebox{\pants}}\big),
F\big(\raisebox{-1mm}{\usebox{\cupcob}}\big),
F\big(\raisebox{-1mm}{\usebox{\copants}}\big),
F\big(\raisebox{0mm}{\usebox{\capcob}}\big)
\bigg).\label{frobbijection}
\end{equation} 
Thus we have a bijection between 2-dimensional TQFTs and commutative Frobenius algebras. Kock shows in \cite[Theorem 3.3.2]{kock2004frobenius} that the bijection above extends to an equivalence of categories.
\begin{theorem} There is a canonical equivalence of categories $\TQFTR \simeq \cFAR.$
\label{TQFTandFrobenius}
\end{theorem}
The  Frobenius algebras/TQFTs in our construction need to satisfy the following additional property, in order for diamonds in the diagram $F^{\vec{x}}\vsdiag$ constructed in Section \ref{Sectionthree} to be commutative.
\begin{defn} \label{specialFA} A Frobenius algebra $(A,\mu,\eta,\Delta,\varepsilon)$ is called \textit{special} if $\mu\circ \Delta=1_A$. 
\end{defn}
 \begin{figure}[h]
 \centering
\begin{tikzpicture}
\useasboundingbox (3,-.6) rectangle (3,.6);
\node at (0,0){$F\left(\usebox\specialmudelta\right)=F\left(\ \usebox\specialcylinder\ \right)$};
\node at (3,0){$\iff$};
\node at (5,0){$\mu\circ\Delta=1_A$};
    \end{tikzpicture} 
    \caption{The defining relation of a special Frobenius algebra (Definition \ref{specialFA}), expressed in terms of the corresponding TQFT. One may interpret this as the fact that special TQFTs do not detect genus (see Lemma \ref{specialTQFT}).}
    \label{mudelta}
    \end{figure}
Figure \ref{mudelta} provides the topological interpretation of this condition. Special Frobenius algebras are also sometimes referred to as \textit{strongly separable} algebras. Aguiar uses this terminology in \cite{aguiar2000note} where he characterizes and provides many nice examples of these algebras.
\begin{ex}
Define $A_n$ to be the $n$-fold product
$$A_n=\mathbb{Z}_2\times \mathbb{Z}_2\times ... \times \mathbb{Z}_2.$$
Then $A_n$ is a $\mathbb{Z}_2$-algebra of dimension $n$ with addition and multiplication defined pointwise.  Consider the basis $\{\vec{e}_i \ | \ 1\leq i\leq n\}$ where $\vec{e}_i=(0,...,1,...,0)$ has all entries 0 except for a single 1 in position $i$. Define the counit $\varepsilon:A_n\rightarrow \mathbb{Z}_2$ by sending $\vec{a}\mapsto \sum_{i=1}^n a_i$
where $\vec{a}=(a_1,...,a_n)$. This determines the comultiplication $\Delta:A_n\rightarrow A_n\otimes A_n$ as given by the formula $\vec{a}\mapsto \sum_{i=1}^n a_i\vec{e}_i\otimes \vec{e}_i$.
These maps endow $A_n$ with the structure of a Frobenius algebra. Actually, since
$$(\mu\circ\Delta)(\vec{a})=\mu\bigg(\sum_{i=1}^n a_i\vec{e}_i\otimes \vec{e}_i\bigg)= \sum_{i=1}^n a_i\vec{e}_i=\vec{a},$$
\noindent $A_n$ is a special Frobenius algebra. Thus we have a family $\{A_n\}_{n\in\Z_+}$ of special Frobenius algebras over $\mathbb{Z}_2$, one for each positive integer $n$. 
\label{special2}
\end{ex}

We now restrict our attention to TQFTs which correspond to special Frobenius algebras via the bijection (\ref{TQFTandFrobenius}). We will make use of the following fact in Section \ref{Sectionthree} during our construction of the complex $\vcplx$.

\begin{lemma} If $F$ is a 2-dimensional TQFT which corresponds to a special Frobenius algebra via (\ref{TQFTandFrobenius}), then for any two connected cobordisms $B,C\in \Hom(M,N)$ with $M,N\in\Cob_2$, we have $F(B)=F(C)$. \label{specialTQFT}

\end{lemma}
\begin{proof}
Given a commutative special Frobenius algebra $A$, let $F$ be the corresponding TQFT. In Figure \ref{mudelta} we see the defining relation $\mu\circ\Delta=1$ expressed topologically. It follows that applying $F$ to any connected cobordism with one incoming and one outgoing boundary also results in the identity map, because any such cobordism can be expressed by stacking the cobordism \raisebox{-1mm}{\usebox\specialcob} on itself an appropriate number of times (just enough to get the correct genus). 
    
Now, take objects $M,N\in\Cob_2$, connected cobordisms $C,D\in \Hom(M,N)$ and suppose $M$ (respectively $N$) consists of $k$ (respectively $\ell$) disjoint circles. Then $C$ and $D$ both have $k$ incoming boundary components and $l$ outgoing boundary components. Since $C$ and $D$ are connected cobordisms in $\Cob_2$, they can be written in \textit{normal form}. That is, write $C=C_1\circ C_2\circ C_3$ where $C_3$ has $k$ incoming boundary components, one outgoing boundary component and genus 0, $C_2$ has one incoming boundary component, the same genus as $C$ and one outgoing boundary component, and $C_1$ has one incoming boundary component, $l$ outgoing boundary component, and genus $0$. Similarly, we can write $D=D_1\circ D_2\circ D_3$ where $D_1=C_1$ and $D_3=C_3$ and $D_2$ has one incoming and one outgoing boundary component and has the same genus as $D$. It follows from the previous paragraph that $F(C_2)=F(D_2)=1$ so 
$$F(C)=F(C_1)\circ F(C_3)=F(D_1)\circ F(D_3)=F(D).
$$
\end{proof}

\section{Categorifying Generalized Vandermonde Determinants}
\label{mainresult}

\subsection{Colored Cobordisms and TQFTs} \label{coloredcobcat}
\label{sec4}

In this section, we introduce colors into the cobordism category, with different colors corresponding to the different variables present in $V(\vec{x})$.

\begin{defn}
Let $S$ be a set, endowed with the discrete topology. Define $\Cob_k^S$ to be the category whose objects are pairs $(M,\phi_M)$ where $M$ is a $(k-1)$-dimensional closed oriented manifold and $\phi_M$ is a continuous map from $M$ to $S$. Morphisms between $(M,\phi_M)$ and $(N, \phi_N)$ are pairs $([A], \phi_A)$ where $[A]$ is an oriented cobordism class from $M$ to $N$ ($[\cdot]$ denotes the diffeomorphism class of $A$) and $\phi_A$ is a continuous map from $A$ to $S$ such that $\phi_A|_{\partial_iA}=\phi_M$ and $\phi_A|_{\partial_oA}=\phi_N$, where $\partial_i$ and $\partial_o$ denote the in-boundary and out-boundary respectively.  
\label{colcobcat}
\end{defn}

Intuitively, the objects of $\Cob_k^S$ are closed  $(k-1)$-manifolds whose connected components are colored by elements of $S$ and morphisms are (diffeomorphism classes of) $k$-manifolds with boundary and with each connected component given a color from $S$. 
\begin{conv} Let $\vec{x}=(x_1,\dots,x_n)$. We introduce the slight abuse of notation $\Cob_k^{\vec{x}}=\Cob_k^{\{x_1,\dots,x_n\}}.$
\end{conv}
For convenience, instead of labeling connected components of manifolds with elements of the set, we display them in different colors. Our running example (starting with Example \ref{exnonex}) is in the case $S=\{x_1,x_2,x_3\}\subseteq \Z_+$ with the convention $x_1={\color{red} \text{red}}, \ x_2={\color{mygreen} \text{green}}, \ x_3={\color{blue} \text{blue}}$. 

\begin{ex}
Let $\vec{x}=(x_1,x_2,x_3)$ have distinct entries and consider the category $\Cob_2^{\vec{x}}$. Let $M= {\color{red} \mathbf{\Circle}}
 {\color{red} \Circle}
 {\color{red} \Circle}
 {\color{mygreen} \Circle}
 {\color{mygreen} \Circle}
 {\color{blue} \Circle}$
 and 
 $N= 
 {\color{red} \Circle}
 {\color{mygreen} \Circle}
 {\color{mygreen} \Circle}
 {\color{blue} \Circle}
 {\color{blue} \Circle}
 {\color{blue} \Circle}$. Figure \ref{fig:3a} shows a colored cobordism from $M$ to $N$ and in Figure \ref{fig:3b} we see an example of a cobordism for which no coloring gives a morphism in $\Cob^{\vec{x}}_2$. 
 \begin{figure}[h]
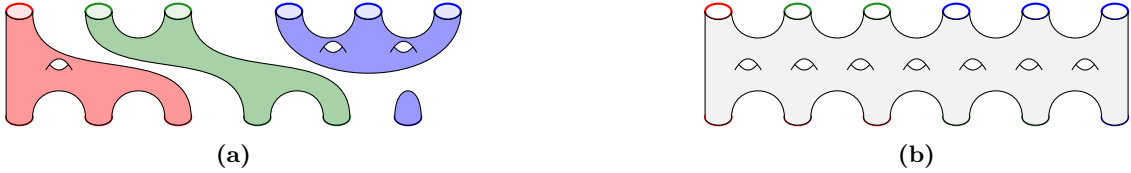

\begin{subfigure}{0.45\textwidth}
\centering
\usebox\coloredcob
\caption{} \label{fig:3a}
\end{subfigure}
\hspace*{\fill} 
\begin{subfigure}{0.45\textwidth}
\centering
\usebox\notcoloredcob
\caption{} \label{fig:3b}
\end{subfigure}
\caption{A colored cobordism in (a) and a non-example in (b).} \label{fig:3}
\end{figure}
\label{exnonex}
\end{ex}
$\Cob_k^S$ is a monoidal category under the operation of taking disjoint unions. Let $\Cob_k^S(x)$ be the subcategory of $\Cob_k^S$ consisting of all objects and morphisms labeled $x$, where $x\in S$. Then 
\[\Cob_k^S\cong \prod_{x\in S}\Cob_k^S(x)\]
where each $\Cob_k^S(x)$ is isomorphic to $\Cob_k$. 
\begin{defn} A \textit{colored topological quantum field theory} is a monoidal functor $F:\Cob_k^S\rightarrow \Rmod$ for some finite set $S$ and a commutative ring $R$ such that the restriction $F_x$ of $F$ to the subcategory ${\Cob_k^S(x)}$ is a TQFT for each $x\in S$. In the case that $k=2$, call a colored TQFT \textit{special} if for each $x\in S$, $F_x$ corresponds to a special Frobenius algebra. For brevity we will use the abbreviation \textit{SC-TQFT} to refer to a \textit{special colored TQFT}.
\label{SCTQFTdef}
\end{defn}
The categorification presented in Section \ref{Sectionthree} takes a sequence $\vec{x}\in\Z_+^n$, and a link diagram $D$ with $n$ crossings, and produces a cohomology theory $\mathcal{H}^*(D,\vec{x})$ whose Euler characteristic is an evaluation of a generalized Vandermonde determinant. This is accomplished in two steps. First we pass from the input data $(D,\vec{x})$ to a diagram $\vsdiag$ in $\Cob_2^{\vec{x}}$ in the shape of the Bruhat order on $S_n$. Second we apply an SC-TQFT to pass to a diagram in the category of vector spaces, from which we obtain $\mathcal{H}^*(D,\vec{x})$. In the remainder of this section we describe the first step: given $\vec{x}\in\Z_+^n$ and a link diagram $D$ with $n$ ordered crossings, construct a functor from the poset category $S_n$ (with the Bruhat order) to $\Cob_2^{\vec{x}}$. That is, for each $\pi\in S_n$ we construct an object $D^\pi_{\vec{x}}\in\Ob(\Cob_2^{\vec{x}})$, and for each cover relation $\pi\lessdot\sigma$ we construct a cobordism $C^{\pi,\sigma}_{\vec{x}}\in\Hom(D^{\pi}_{\vec{x}},D^{\sigma}_{\vec{x}}).$
Recall that crossings in knot diagrams have two types of smoothings, the \textit{0-smoothing} and the \textit{1-smoothing} (see Figure \ref{zeroandone}). 

\begin{figure}[h]
\centering
\begin{tikzpicture}[scale=.6]
\node (1) at (0,0){\usebox\crossingdiag};
\node (2) at (4,0){\usebox\zerosmoothdiag};
\node (3) at (-4,0){\usebox\onesmoothdiag};
\draw[->,decorate,decoration={snake,amplitude=.4mm,segment length=3mm}] (1,0)--(3,0);
\draw[->,decorate,decoration={snake,amplitude=.4mm,segment length=3mm}] (-1,0)--(-3,0);
\node at (2,.4) {1};
\node at (-2,.4) {0};
\end{tikzpicture}
\caption{The 0 and 1 smoothings of a crossing in a knot diagram.}
\label{zeroandone}
\end{figure}
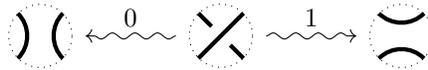

\begin{defn} Fix $\vec{x}=(x_1,\dots,x_n)\in \Z_+^n$ and let $D$ be a link diagram with $n$ crossings with a choice of a total ordering $c_1,...,c_n$. Given $\pi\in S_n$, the \textit{$\pi$-smoothing} of $D$ (relative to $\vec{x}$) is
$$D^\pi_{\vec{x}}=D^\pi_{x_1} \amalg D^\pi_{x_2} \amalg ... \amalg D^\pi_{x_n} \in \Cob_2^{\vec{x}}$$
where $D^\pi_{x_i}\in \Cob_2^{\vec{x}}(x_i)$ is gotten by giving crossings $c_1,c_2,\dots,c_{\pi(i)}$ the 1-smoothing, giving all other crossings the 0-smoothing, and giving all connected components of $D^\pi_{x_i}$ the label $x_i$. Let $s_i^\pi(D)$ denote the number of disjoint circles in $D_{x_i}^\pi$. For simplicity we often write $D^\pi=D^\pi_{\vec{x}}$ when no confusion will arise from doing so.
\end{defn}
See Figure \ref{smoothing} for an example in the case $D=T_{2,3}$.
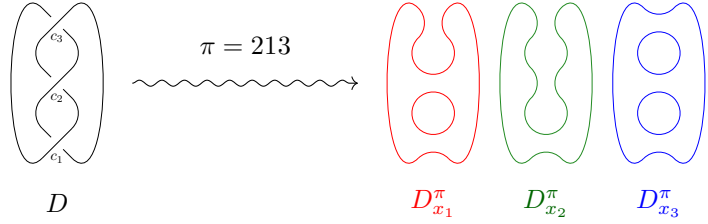
\begin{figure}[h]
\centering
\begin{tikzpicture}
\useasboundingbox (-.7,-.8) rectangle (8.7,2.2);
\node[scale=.8] at (0,1) {\usebox\postreflabeled};
\node at (0,-.6) {$D$};
\draw[->,decorate,decoration={snake,amplitude=.4mm,segment length=3mm}] (1,1)--(4,1);
\node at (2.5,1.5){$\pi=213$};
\node[scale=.4] at (5,1) {\usebox\trefBBAred};
\node at (5,-.6) {${\color{red} D^{\pi}_{x_1}}$};
\node[scale=.4] at (6.5,1) {\usebox\trefBAAgreen};
\node at (6.5,-.6) {${\color{green!40!black} D^{\pi}_{x_2}}$};
\node[scale=.4] at (8,1) {\usebox\trefBBBblue};
\node at (8,-.6) {${\color{blue} D^{\pi}_{x_3}}$};
\end{tikzpicture}
\caption{The $213$-smoothing of $T_{2,3}.$}
\label{smoothing}
\end{figure}
\begin{defn} Fix $\vec{x}=(x_1,\dots,x_n)\in \Z_+^n$ and let $D$ be a link diagram with $n$ crossings with a choice of a total ordering $c_1,...,c_n$. Suppose $\sigma\in S_n$ covers $\pi$ in the Bruhat order. Define the colored cobordism $C^{\pi,\sigma}_{\vec{x}}$ from $D^\pi_{\vec{x}}$ to $D^\sigma_{\vec{x}}$ in the category $\Cob_2^{\vec{x}}$ by 
$$C^{\pi,\sigma}_{\vec{x}}=C^{\pi,\sigma}_{x_1} \amalg ... \amalg C^{\pi,\sigma}_{x_n}$$
where $C^{\pi,\sigma}_{x_i}$ is the identity cobordism from $D^\pi_{x_i}$ to $D^\sigma_{x_i}$ if $\pi(i)=\sigma(i)$, or the unique connected genus zero cobordism from $D^\pi_{x_i}$ to $D^\sigma_{x_i}$ if $\pi(i)\neq \sigma(i)$. The cobordism $C^{\pi,\sigma}_{x_i}$ is given the color $x_i$. Again, for simplicity we often write $C^{\pi,\sigma}_{\vec{x}}=C^{\pi,\sigma}$ when no confusion will arise from doing so.
\end{defn}
\begin{conv}To avoid having to always draw 2-dimensional pictures of cobordisms, we can use the following shortcut. To denote the connected genus zero cobordism from a smoothing $D_{x_i}^\pi$ to another smoothing $D_{x_i}^\sigma$, simply circle which crossings change in the picture of $D_{x_i}^\pi$. See Figure \ref{shortcut} for an example. Warning: since we require the cobordisms of each color to be connected, the circles around crossings do not indicate local saddle cobordisms (like one might expect in comparison to Bar-Natan's notation in \cite{bar2002khovanov}).
\label{cobordismconvention}\end{conv}
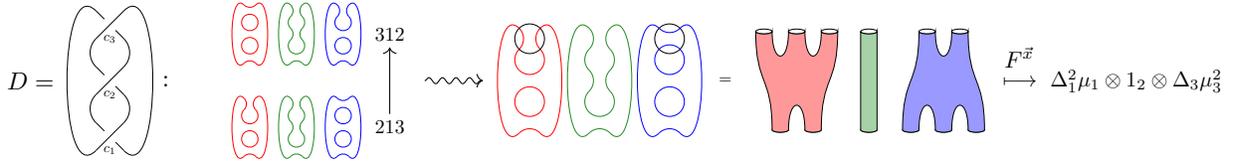
\begin{figure}[h]
\centering
\begin{tikzpicture}[scale=.62, every node/.style={transform shape}]
\node[scale=1.5] at (-6.2,0){$D=$};
\node[scale=1.2] at (-4.5,0){\usebox\postreflabeled};
\node[scale=1.5] at (-3.3,0){:};
\node[scale=.25] at (-1.5,-1){\usebox\trefBBAred};
\node[scale=.25] at (-.5,-1) {\usebox\trefBAAgreen};
\node[scale=.25] at (.5,-1) {\usebox\trefBBBblue};
\node[scale=.25] at (-1.5,1){\usebox\trefBBBred};
\node[scale=.25] at (-.5,1) {\usebox\trefBAAgreen};
\node[scale=.25] at (.5,1) {\usebox\trefBBAblue};
 \node[scale=1.2] (a) at (1.5,-1) {213};
        \node[scale=1.2] (b) at (1.5,1) {312};
        \draw[->] (a)--(b);
        \draw[->,decorate,decoration={snake,amplitude=.4mm,segment length=1.8mm}] (2.25,0)--(3.5,0);
\node[scale=.45] at (4.5,0) {\usebox\trefBBAredCircThree};
\node[scale=.45] at (6,0) {\usebox\trefBAAgreen};
\node[scale=.45] at (7.5,0) {\usebox\trefBBBblueCircThree};
\node at (8.7,0){$=$};
\node at (11.8,0){\usebox\coloredcobo};
\node[scale=1.3] at (17,0){$\longmapsto \ \Delta_1^2\mu_1\otimes 1_2\otimes \Delta_3\mu_3^2$ };
\node[scale=1.4] at (15,.5){$F^{\vec{x}}$};
\end{tikzpicture}
\caption{A shortcut notation for colored cobordisms, as explained in \ref{cobordismconvention}. The above picture denotes the cobordism $C^{213,312}$ from $D^{213}$ to $D^{312}$ where $D=T_{2,3}$.}
\label{shortcut}
\end{figure}
\subsection{Construction of the Chain Complex $\mathcal{C}^*(D,\vec{x})$}
\label{Sectionthree}

\begin{defn}
Let $D$ be a link diagram with $n$ crossings with a fixed total ordering of the crossings $c_1,c_2,\dots,c_n$, and let $\vec{x}\in \Z_+^n$. Construct a diagram $[\![D]\!]$ in $\Cob_2^{\vec{x}}$ according to the following steps:
\begin{enumerate}[label=\arabic*.]
\item Start with the Hasse diagram of the (strong) Bruhat order on $S_n$. 
\item Replace each vertex $\pi\in S_n$ by the $\pi$-smoothing $D^\pi$. 
\item For each cover relation $\pi\lessdot\sigma$, replace the edge from $\pi$ to $\sigma$ by the cobordism $C^{\pi,\sigma}$. 
\end{enumerate}
\end{defn}
See Figure \ref{VandHom} for an example in the case that $D=T_{2,3}$. 
\begin{figure}[h!]
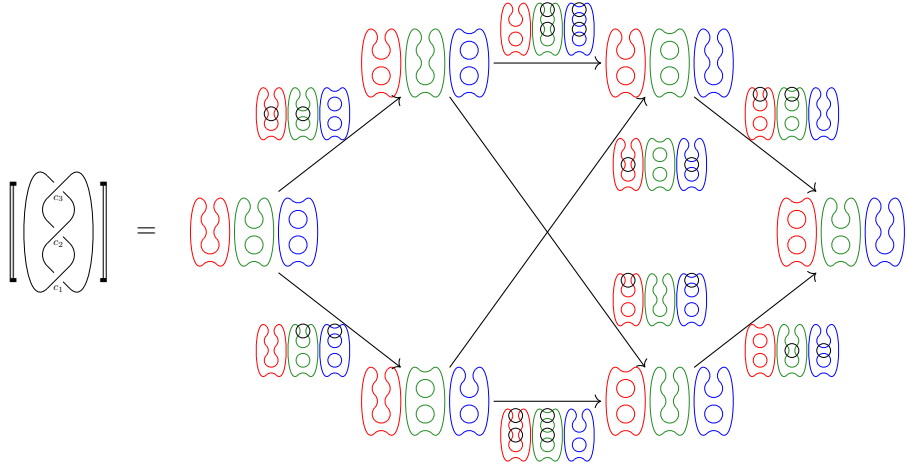

\centering
\usebox\VandSmoothCplxd
\caption{Shown above is the diagram $[\![D]\!]$ in $\Cob_2^{\vec{x}}$ in the case that $D$ is a diagram of the torus knot $T_{2,3}$ with crossings ordered from bottom to top, and $\vec{x}=(x_1,x_2,x_3)$ has distinct entries.}
\label{VandHom}
\end{figure}
\begin{defn} Let $D$ be a link diagram with $n$ crossings and with a fixed total ordering of the crossings $c_1,c_2,\dots,c_n$. Let $\vec{x}=(x_1,...,x_n)\in \mathbb{Z}_+^n$ and let $F^{\vec{x}}$ be a 2-dimensional SC-TQFT such that for $i\in[n]$, the restriction $F^{\vec{x}}_{x_i}$ corresponds to the special Frobenius algebra $A_{x_i}$ of dimension $x_i$ over $\mathbb{Z}_2$ defined in Example \ref{special2}. Define the cochain groups
$$\mathcal{C}^k(D,\vec{x})=\bigoplus_{\substack{\pi\in S_n\\ \text{inv}(\pi)=k}} F^{\vec{x}}(D^\pi).$$
Define a degree 1 codifferential $\delta^k:\mathcal{C}^k\rightarrow \mathcal{C}^{k+1}$ by 
$$\delta^k(v)=\bigoplus_{\sigma\gtrdot \pi} F^{\vec{x}}(C^{\pi,\sigma})(v)$$
for $v\in F^{\vec{x}}(D^\pi)$ with $\text{inv}(\pi)=k$. Let $\mathcal{H}^*(D,\vec{x})$ denote the cohomology groups of $\mathcal{C}^*(D,\vec{x})$. 
\label{SCTQFT} 
\end{defn}
Lemma \ref{vandcomplexlemma} shows that $\delta^2=0$ so this is indeed a cochain complex. Figure \ref{cochaincomplex} provides an illustration of the construction of $\mathcal{C}^*(D,\vec{x})$ for our running example with the diagram $T_{2,3}$. 

\begin{figure}[h]
\centering
\usebox\VandCplx
\caption{The complex $\mathcal{C}^*(D,\vec{x})$ is gotten from $[\![D]\!]$ by applying the SC-TQFT $F^{\vec{x}}$ and taking direct sums down ranks. Here we show the result for the case of the diagram $T_{2,3}$, continuing from Figure \ref{VandHom}.}
\label{cochaincomplex}
\end{figure}

\begin{notn}
Let $\mu_i$ and $\Delta_i$ denote the multiplication and comultiplication in $F^{\vec{x}}_{x_i}(\Circle_{x_i})=A_{x_i}$. As a slight abuse of notation, if a connected cobordism $C_r^l$ of color $i$ has $r$ incoming and $l$ outgoing boundary components, we will write its image under $F^{\vec{x}}$ as $F^{\vec{x}}(C_r^l)=\Delta_i^{l-1}\mu_i^{r-1}$, indicating that we multiply the $r$ tensor factors $A_{x_i}^{\otimes r}$ by multiplying $r-1$ times and then comultiply out to $l$ tensor factors $A_{x_i}^{\otimes l}$ by comultiplying $l-1$ times. This notation is well defined due to the associativity and coassociativity of the Frobenius algebra $A_{x_i}$, and the fact that we are using a special Frobenius algebra (so the genus of $C_r^l$ is irrelevant). We often omit subscripts on $\mu$ and $\Delta$ when it is painfully clear what the subscript should be by context. See Figure \ref{shortcut} for an example. 
\end{notn}
\begin{rmk} Since the Bruhat order is the face poset of a CW complex (see for example \cite[Corollary 2.7.14]{bjorner2006combinatorics}), there is a $\{+1,-1\}$-coloring of the edges of the Hasse diagram for which each diamond has an odd number of $-1$'s. Thus, the complex $\mathcal{C}^*(D,\vec{x})$ can be constructed over $\mathbb{Z}$ instead of $\mathbb{Z}_2$ by weighting per-edge maps with these signs. However more work needs to be done to make this explicit.
\label{zcoeff}
\end{rmk}

\begin{lemma} For any link diagram $D$ and any $\vec{x}\in\mathbb{Z}_+^n$, $\big(\mathcal{C}^*(D,\vec{x}),\delta\big)$ is a cochain complex.
\label{vandcomplexlemma}
\end{lemma}
\begin{proof} Since each interval of length 2 in the Bruhat order is a diamond, it suffices to show that the codifferential commutes on each diamond (since our algebras are over $\mathbb{Z}_2$, commuting and anticommuting are equivalent). Consider a diamond in the Bruhat order and the corresponding diamond in $[\![D]\!]$ (shown in Figure \ref{fig:1}a and \ref{fig:1}b). We proceed by analyzing the following cases:
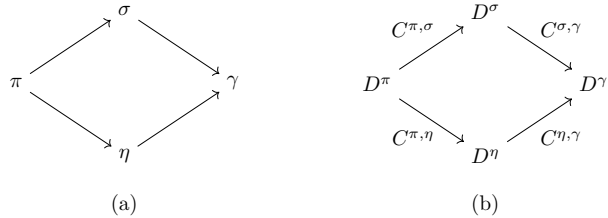
\begin{figure}[h]
\centering
\begin{tikzpicture}[scale=0.8, every node/.style={transform shape}]
\node at (2,2) {\usebox\bruhatdiamond};
\node at (2,0){(a)};
\node at (8,0){(b)};
\node at (8,2) {\usebox\diagramdiamond};
\end{tikzpicture}
\caption{A diamond in the Bruhat order and the corresponding diamond in the diagram $[\![D]\!]$.} \label{fig:1}
\end{figure}

\begin{case} {\it $\pi$ and $\gamma$ differ in 3 positions, $i,j$ and $k$ (in one-line notation)}. 
Without loss of generality we can write $\sigma=\pi (i,j)$ and $\gamma=\sigma (i,k)$. Now, $(i,j)(i,k)=(i,k,j)$ can be written in exactly two other ways as a product of transpositions, $(i,k)(j,k)=(i,k,j)=(j,k)(i,j)$ so our diamond is given by either Figure \ref{fig:2}a or \ref{fig:2}b.

\begin{figure}[h!]
\centering
\begin{tikzpicture}[scale=0.8, every node/.style={transform shape}]
\node at (2,2) {\usebox\diamondone};
\node at (2,0){(a)};
\node at (8,0){(b)};
\node at (8,2) {\usebox\diamondtwo};
\node at (14,2){\usebox\diamondthree};
\node at (14,0){(c)};
\end{tikzpicture}
\caption{In (a) and (b) we see the two possible diamonds which can appear in case 1. In (c) we see the only possible type of diamond which can appear in case 2.} \label{fig:2}
\end{figure}
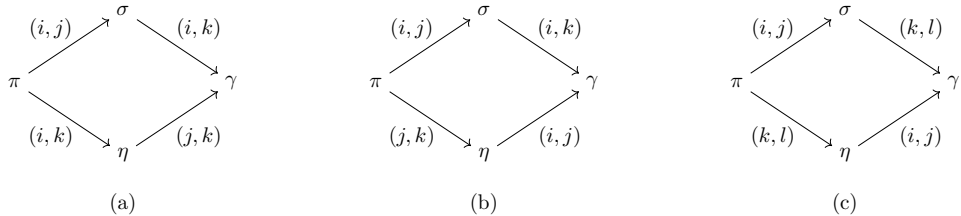

First consider the diamond in Figure \ref{fig:2}a. By construction, $C^{\pi,\sigma}_{x_i}$ is a connected genus zero cobordism from $D^\pi_{x_i}$ to $D^\sigma_{x_i}$ and $C_{x_i}^{\sigma,\gamma}$ is a connected genus zero cobordism from $D^\sigma_{x_i}$ to $D^\gamma_{x_i}$. Thus $C_{x_i}^{\sigma,\gamma}\circ C^{\pi,\sigma}_{x_i}$ is a connected cobordism from $D_{x_i}^\pi$ to $D_{x_i}^\gamma$. Going the other way, $C^{\pi,\eta}_{x_i}$ is a connected genus zero cobordism from $D^\pi_{x_i}$ to $D^\eta_{x_i}$ and $C_{x_i}^{\eta,\gamma}$ is the identity cobordism from $D^\eta_{x_i}$ to $D^\gamma_{x_i}$. Thus $C_{x_i}^{\eta,\gamma}\circ C^{\pi,\eta}_{x_i}$ is also a connected cobordism from $D_{x_i}^\pi$ to $D_{x_i}^\gamma$. Thus by Lemma \ref{specialTQFT}, $F^{\vec{x}}(C_{x_i}^{\sigma,\gamma}\circ C^{\pi,\sigma}_{x_i})=F^{\vec{x}}(C_{x_i}^{\eta,\gamma}\circ C^{\pi,\eta}_{x_i})$. The same argument works for the pieces colored $j$ and $k$. For $l\notin \{i,j,k\}$ things are even easier, as $C_{x_l}^{\sigma,\gamma}\circ C^{\pi,\sigma}_{x_l}=C_{x_l}^{\eta,\gamma}\circ C^{\pi,\eta}_{x_l}$ is by definition just the identity cobordism. A similar argument applies to the diamond in Figure \ref{fig:2}b.
\end{case}

\begin{case} {\it $\pi$ and $\gamma$ differ in 4 positions $i,j,k$ and $l$ (in one-line notation)}. Then our diamond looks like one shown in Figure \ref{fig:2}c. A similar argument works here as in case 1. 
\end{case}

Thus each diamond in $F^{\vec{x}}[\![D]\!]$ commutes. Now, for any $x\in \mathcal{C}^*(D,\vec{x})$ each contribution to $\delta^2(x)$ comes from a chain of length 2 in $F^{\vec{x}}[\![D]\!]$ and has an identical contribution coming from the other half of the diamond. Thus for any $x\in \mathcal{C}^*(D,\vec{x})$, all terms in $\delta^2(x)$ have even coefficients.  Since our algebras have coefficients in $\mathbb{Z}_2$, we find that $\delta^2=0$. 
\end{proof}

We now specialize our construction to the 2-torus link diagrams $T_{2,n}$. Recall $T_{2,n}$ denotes the closure of the diagram $\sigma_1^n$ in the braid group $B_2$ where $\sigma_1=\usebox\inlinecrossing \ $ (see figure \ref{fig:torusa}). It turns out (by Lemma \ref{height} and Theorem \ref{KirkMain}) that smoothings of the diagram $T_{2,n}$ have the appropriate combinatorics needed to categorify the Vandermonde determinant $V(\vec{x})$.
%
%
\begin{lemma}
Let $c_1,c_2,\dots, c_n$ be a total ordering of the crossings of $T_{2,n}$ and consider $\vec{\alpha}=(\alpha_1,\dots,\alpha_n)\in\{0,1\}^n$. Let $h(\vec{\alpha})=\sum_i\alpha_i$ denote the height of $\vec{\alpha}$. Then if $h(\vec{\alpha})>0$, the smoothing of $T_{2,n}$ corresponding to $\vec{\alpha}$ consists of $h(\vec{\alpha})$ disjoint circles.
\label{height}
\end{lemma}
\begin{proof} 
\begin{figure}[h]
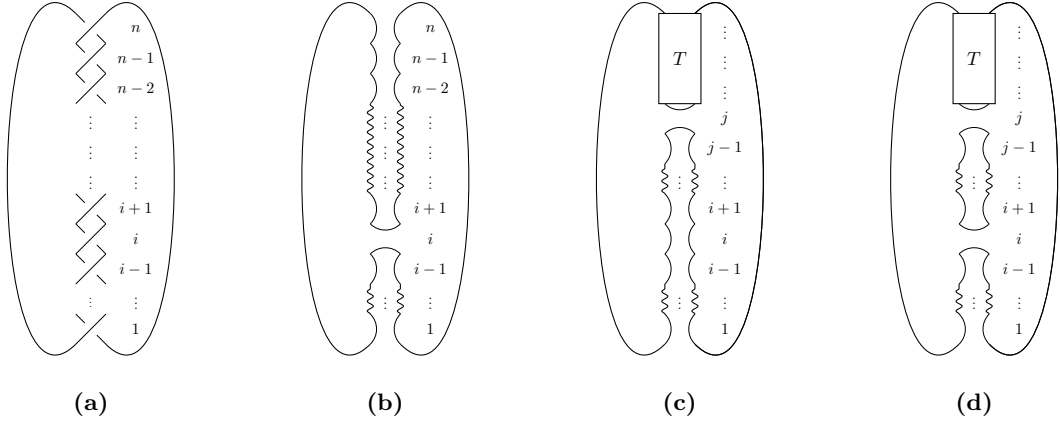

\begin{subfigure}{0.23\textwidth}
\centering
\usebox\torusproof
\caption{} \label{fig:torusa}
\end{subfigure}
\begin{subfigure}{0.23\textwidth}
\centering
\usebox\torusheightone
\caption{} \label{fig:torusb}
\end{subfigure}
\begin{subfigure}{0.23\textwidth}
\centering
\usebox\torusheightiminus
\caption{} \label{fig:torusc}
\end{subfigure}
\begin{subfigure}{0.23\textwidth}
\centering
\usebox\torusheighti
\caption{} \label{fig:torusd}
\end{subfigure}
\caption{(a) The diagram $T_{2,n}$. (b) A smoothing of $T_{2,n}$ with height 1. (c) The smoothing $\vec{\alpha}^\prime$ gotten by changing the first $1$ in $\vec{\alpha}$ to 0. (d) The smoothing $\vec{\alpha}$ from the proof of Lemma \ref{height}.}
\label{torussmoothingsdiagram}
\end{figure}
We proceed by induction on $h(\vec{\alpha})$. If $h(\vec{\alpha})=1$, then $\alpha_i=1$ for some $i\in [n]$ and $\alpha_j=0$ for each $j\neq i$. The smoothing corresponding to $\vec{\alpha}$ is shown in Figure \ref{fig:torusa} and consists of one circle. 

Now suppose that $h(\vec{\alpha})>1$, and choose $i$ minimally such that $\alpha_i=1$ and $j$ minimally such that $j>i$ and $\alpha_j=1$. Consider the tuple $\vec{\alpha}^\prime$ for which $\alpha_j^\prime=\alpha_j$ for $j\neq i$ and $\alpha^\prime_i=0$. The smoothings above the $j^{\text{th}}$ crossing form a trivial tangle which we will denote by $T$, and this tangle is the same in both the smoothing corresponding to $\vec{\alpha}$ and the smoothing corresponding to $\vec{\alpha}^\prime$. By induction, the smoothing corresponding to $\vec{\alpha}^\prime$ consists of $h(\vec{\alpha}^{\prime})=h(\vec{\alpha})-1$ disjoint circles. As shown in Figure \ref{fig:torusc} and \ref{fig:torusd}, changing the $i^{\text{th}}$ smoothing from 0 to 1 splits one circle into two. Thus the smoothing corresponding to $\vec{\alpha}$ consists of $h(\vec{\alpha})$ disjoint circles.
\end{proof}
\begin{theorem} The Euler characteristic of  $\mathcal{H}^*(T_{2,n},\vec{x})$ is equal to  the Vandermonde determinant $V(\vec{x})$:
$$\sum_{k\geq 0}(-1)^k\dim\tvkhcplx=V(\vec{x}).$$
Furthermore, the complex $\mathcal{C}^*(T_{2,n},\vec{x})$ does not depend on the ordering of the crossings of $T_{2,n}$. 
\label{KirkMain}
\end{theorem}
\begin{proof} By Lemma \ref{height}, the smoothing $(T_{2,n})^{\pi}_{x_i}$ consists of $s_i^{\pi}(T_{2,n})=\pi(i)$ circles of color $x_i$ and thus the algebra $F^{\vec{x}}(T_{2,n})^\pi_{\vec{x}}$ has dimension $x_1^{\pi(1)}x_2^{\pi(2)}\dots x_n^{\pi(n)}$,
and so the dimension of $\tvkcplx$ is equal to $\sum_{\inv(\pi)=k} x_1^{\pi(1)}\dots x_n^{\pi(n)}$. Thus we have
\begin{align*}
\sum_{k\geq 0} (-1)^k \dim \tvkhcplx
&=
\sum_{k\geq 0} (-1)^k \dim \tvkcplx \\
&=\sum_{k\geq 0} \left(\sum_{\text{inv}\pi=k}(-1)^{\text{inv}\pi} x_1^{\pi(1)}x_2^{\pi(2)}...x_n^{\pi(n)}\right)=V(\vec{x}).
\end{align*}
In fact, it follows immediately from Lemma \ref{height} and the definition of $\mathcal{C}^*(T_{2,n},\vec{x})$ that under any change in the ordering of the crossings, the complex $\mathcal{C}^*(T_{2,n},\vec{x})$ itself remains unchanged.
\end{proof}
%

The Theorem  \ref{KirkMain} states that using the 2-strand torus link diagram $T_{2,n}$ in our construction yields a categorification of the Vandermonde determinant. In case we remove the requirement about the knot and diagram type, Theorem \ref{genMain} states that given an arbitrary link diagrams our construction provides a categorifications of certain generalized Vandermonde determinants. 

\begin{theorem} Let $D$ be a link diagram  with a fixed ordering $c_1,\dots, c_n$ of the crossings. Let $s_k=s^{\textnormal{id}}_k(D)$ denote the number of circles in the smoothing of $D$ where crossings $c_1,\dots,c_k$ are given 1-smoothings and all others are given 0-smoothings. Then the Euler characteristic of $\vhcplx$ is equal to the generalized Vandermonde determinant:
\[\sum_{k\geq 0}(-1)^k\dim\vkhcplx = \det (x_i^{s_j})_{i,j=1}^n.\]
Furthermore, if $D$ has the property that the height of a smoothing determines the number of circles in the smoothing, then $\mathcal{C}^*(D,\vec{x})$ does not depend on the ordering of the crossings.\label{genMain}
\end{theorem}
\begin{proof} The proof is the same as the proof of Theorem \ref{KirkMain} with the exception that the dimension of the algebra $F^{\vec{x}}(D^\pi)$ is now equal to $x_1^{s_{\pi(1)}}\dots x_n^{s_{\pi(n)}}$. 
\end{proof}
\subsection{Functoriality}
\label{functsection}
Given $\vec{x}\in\Z_+^n$ and a link diagram $D$ with $n$ crossings, one obtains a generalized Vandermonde determinant which we will denote $V_D(\vec{x})=\det(x_i^{s_j})_{i,j=1}^n$ as determined by Theorem \ref{genMain}. Thus the diagram $D$ determines a function $V_D:\Z_n^+\rightarrow \mathbb{Z}$ sending $\vec{x}\mapsto V_D(\vec{x})$. In Section \ref{mainresult}, we upgraded $V_D$ to a function $\Z_n^+\rightarrow \Ob(\grab)$ to the objects of the category of graded abelian groups, via $\vec{x}\mapsto \vhcplx$. In this section, we categorify this function to obtain a functor. We identify each $\vec{x}\in\Z_+^n$ as a collection of $n$ dots on a vertical axis labeled by the $x_i$, with the dot corresponding to $x_j$ located at the coordinate $j$. We construct a category $\Zndiag$ with objects $Z_n^+$ and morphisms as certain diagrams (see Definition \ref{functdiagrams}) connecting these dots. Given any morphism from $\vec{x}$ to $\vec{y}$ in $\Zndiag$ we construct graded homomorphisms $\mathcal{H}^*(D,\vec{x})\to\mathcal{H}^*(D,\vec{y})$, yielding a functor $\mathcal{V}_D:\Zndiag\to\grab$.

\begin{defn}
Define the category $\Zndiag$ to have:
\begin{itemize}
\item Objects $\Ob(\Zndiag)=Z_n^+$
\item Given $\vec{x},\vec{y}\in Z_n^+$ morphisms from $\vec{x}$ to $\vec{y}$ are isotopy classes of diagrams in the infinite strip $[0,1]\times \R$ in the $uv$-plane with $\vec{x}$ located at the boundary $u=0$, $\vec{y}$ at located at the boundary $u=1$, and a collection of arcs connecting a subcollection of points in $\vec{x}$ and a subcollection of points in $\vec{y}$. Arcs between $\vec{x}$ and $\vec{y}$ must start at a point $x_i$ of $\vec{x}$ and end at a point $y_j$ of $\vec{y}$ such that $i\leq j$ and $x_i=y_j$. We also allow the possibility of having a finite number of labeled dots in the interior region $0<u<1$ of a diagram. 
\item Compose morphisms $A\in \Hom(\vec{x},\vec{y})$ and $B\in \Hom(\vec{y},\vec{z})$ by concatenating horizontally along $\vec{y}$ and applying the following:
\begin{enumerate}
\item Rescale in the $u$ direction by a factor of $\frac{1}{2}$ so the resulting diagram again lives in $[0,1]\times\R$.
\item Any arc with an endpoint in the middle of a diagram contracts to the endpoint on the boundary.
\end{enumerate}
\end{itemize}
\label{functdiagrams}
\end{defn}

For an example of a typical composition of morphisms in the category $\Zndiag$, see Figure \ref{fig:16}.
\begin{figure}[h]
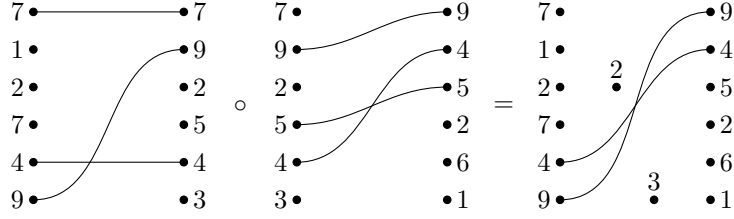

\centering
\usebox\functorialcomp
\caption{A typical composition of morphisms in the category $\Zndiag$ with $n=6$.}
\label{fig:16}
\end{figure}
Given a morphism $A$ in $\Zndiag$ from $\vec{x}=(x_1,\dots,x_n)$ to $\vec{y}=(y_1,\dots,y_n)$, we now construct a chain map from $\vcplx$ to $\mathcal{C}^*(D,\vec{y})$. We do this by first constructing, for each $\pi\in S_n$, a 
colored cobordism $C^\pi(A)$ from $D^\pi_{\vec{x}}$ to 
$D^\pi_{\vec{y}}$ for each $\pi\in S_n$ according to the following rules: 
\begin{itemize}
\item If $x_i$ has no arc connected to it, we use the unique connected genus zero cobordism from $D^\pi_{x_i}$ to $\varnothing$. This cobordism is assigned the color $x_i$.
\item If $i=j$, and there is an arc from $x_i$ to $y_j$, then $D^{\pi}_{x_i}$ is equal to 
$D^{\pi}_{y_j}$ and we connect $D^{\pi}_{x_i}$ to $D^{\pi}_{y_j}$ with the identity cobordism of color $x_i=y_j$. 
\item If $i< j$,  and there is an arc from $x_i$ to $y_j$ then we connect $D^{\pi}_{x_i}$ to $D^{\pi}_{y_j}$ with the unique connected genus zero cobordism of color $x_i=y_j$. 
\item If there is no arc connected to $y_j$, we use the unique connected genus zero cobordism from $\varnothing$ to $D^{\pi}_{y_j}$ of color $y_j$.
\item A colored dot on the interior region of a diagram becomes a sphere in $C^\pi(A)$ of the same color.
\end{itemize}
The next lemma follows directly from the definition of $C^\pi(A)$:
\begin{lemma} Let $I_{\vec{x}}$ denote the identity morphism on $\vec{x}$ in $\Zndiag$. Then 
\begin{enumerate}
\item $C^\pi(I_{\vec{x}})$ is the identity cobordism on $D^{\pi}$.
\item Given composable morphisms $A,B$ in $\Zndiag$, there is a color preserving bijection between connected components of $C^\pi(A\circ B)$ and $C^\pi(A)\circ C^\pi(B)$ which preserves number of in-boundary components and number of out-boundary components (but possibly does not preserve genus).
\end{enumerate}
\label{functoriallemma}
\end{lemma}
\begin{defn}
For each $A\in\Mor(\vec{x},\vec{y})$ define the map $A_{\#}:\vcplx\to\mathcal{C}^*(D,\vec{y})$ via
\[A_{\#}: \ v\mapsto F^{\vec{x}\oplus\vec{y}}\big(C^\pi(A)\big)(v)\]
for $v\in F^{\vec{x}}D^\pi_{\vec{x}}$ and where $F^{\vec{x}\oplus\vec{y}}$ is the SC-TQFT on the category $\Cob_2^{\{x_1,\dots,x_n,y_1,\dots,y_n\}}$ taking $\Circle_{x_i}$ to $A_{x_i}$ and $\Circle_{y_j}$ to $A_{y_j}$ for $i,j\in[n]$. 
\end{defn}

In Lemma \ref{functchainmap} we show that  $A_{\#}$ is a chain map. Thus we get an induced map on cohomology which we denote $\mathcal{V}_D(A):\mathcal{H}^*(D,\vec{x})\to \mathcal{H}^*(D,\vec{y})$.

\begin{rmk} It follows immediately from Lemma \ref{functoriallemma} part 2 that for composable morphisms $A,B$ in $\Zndiag$, we have $F^{\vec{x}\oplus\vec{y}}(A\circ B)=F^{\vec{x}\oplus\vec{y}}(A)\circ F^{\vec{x}\oplus\vec{y}}(B)$.\label{functlemmarmk}
\end{rmk}

\begin{lemma} For each $A\in\Mor(\vec{x},\vec{y})$, $A_{\#}: \mathcal{C}^*(D,\vec{x})\to \mathcal{C}^*(D,\vec{y})$ is a chain map. 
\label{functchainmap}
\end{lemma}
\begin{proof} We must show that for each $\pi\in S_n$ and $v\in F^{\vec{x}}\big(D^\pi_{\vec{x}}\big)$, we have $\big(\delta A_{\#}\big)(v)=\big(A_{\#}\delta\big)(v)$. It suffices to show that for each $\sigma\in S_n$ which covers $\pi$ in the Bruhat order, we have the equality
\begin{equation}
C^\sigma(A)\circ C^{\pi,\sigma}_{\vec{x}}=C^{\pi,\sigma}_{\vec{y}}\circ C^\pi(A).\label{cobfunct}\end{equation}
or at least
\begin{equation}
F^{\vec{x}\oplus\vec{y}}\big(C^\sigma(A)\circ C^{\pi,\sigma}_{\vec{x}}\big)=F^{\vec{x}\oplus\vec{y}}\big(C^{\pi,\sigma}_{\vec{y}}\circ C^\pi(A)\big).\label{weakcobfunct}\end{equation}
Without loss of generality we can assume there are no floating dots in the interior of $A$. Thus there are no closed connected (floating) components of either side of equation (\ref{cobfunct}). It suffices to check for each boundary piece $D^{\pi}_{x_i}$, $D^{\sigma}_{y_j}$ for $i,j\in[n]$, the image under $F^{\vec{x}\oplus\vec{y}}$ is the same. In the following six cases we consider the cobordism attached to $D^{\pi}_{x_i}$ via either side of equation (\ref{cobfunct}):

\begin{case} {\it $\pi(i)=\sigma(i)$, no arc connected to $x_i$ in A}. The cobordism attached to $D^{\pi}_{x_i}$ via either side of equation (\ref{cobfunct}) consists of a connected genus zero cobordism from $D^\pi_{x_i}$ to $\varnothing$. 
\end{case}

\begin{case} {\it  $\pi(i)=\sigma(i)$, $i< j$, arc connected from $x_i$ to $y_j$ in A.} The cobordism attached to $D^{\pi}_{x_i}$ via either side of equation (\ref{cobfunct}) consists of a connected genus zero cobordism from $D^\pi_{x_i}$ to $D^\sigma_{y_j}$. 
\end{case}

\begin{case} {\it $\pi(i)=\sigma(i)$, $i=j$ arc connected from $x_i$ to $y_j$ in A.} The cobordism attached to $D^{\pi}_{x_i}$ via either side of equation (\ref{cobfunct}) consists of the identity cobordism from $D^\pi_{x_i}$ to $D^\sigma_{y_j}$. 
\end{case}

\begin{case} {\it  $\pi(i)\neq\sigma(i)$, no arc connected to $x_i$ in A.} The cobordism attached to $D^{\pi}_{x_i}$ via either side of equation (\ref{cobfunct}) consists of a connected cobordism from $D^\pi_{x_i}$ to $\varnothing$. Thus (\ref{cobfunct}) is not necessarily satisfied but by Lemma \ref{specialTQFT}, (\ref{weakcobfunct}) is.
\end{case}

\begin{case} {\it  $\pi(i)\neq\sigma(i)$, $i< j$ arc connected from $x_i$ to $y_j$ in A.} The cobordism attached to $D^{\pi}_{x_i}$ via either side of equation (\ref{cobfunct}) consists of a connected cobordism from $D^\pi_{x_i}$ to $D^\sigma_{y_j}$. Thus (\ref{cobfunct}) is not necessarily satisfied but by Lemma \ref{specialTQFT}, (\ref{weakcobfunct}) is.
\end{case}

\begin{case} {\it  $\pi(i)\neq\sigma(i)$, $i= j$ arc connected from $x_i$ to $y_j$ in A.} The cobordism attached to $D^{\pi}_{x_i}$ via either side of equation (\ref{cobfunct}) consists of a connected genus zero cobordism from $D^\pi_{x_i}$ to $D^\sigma_{y_j}$. 
\end{case}

It remains to repeat these six cases for the cobordism attached to $D^{\sigma}_{y_j}$ via either side of equation (\ref{cobfunct}). However, by symmetry of the definitions, the proofs in these cases are the same as those shown above. 
\end{proof}

\begin{theorem}
For any link diagram $D$, the correspondence $\mathcal{V}_D:\Zndiag\to \grab$ sending objects $\vec{x}\mapsto \mathcal{V}_D(\vec{x})=\mathcal{H}^*(D,\vec{x})$, and morphisms $A\mapsto \mathcal{V}(A)$, is a functor.
\end{theorem}
\begin{proof}
By Lemma \ref{functoriallemma} part 1, $\mathcal{V}_D(I_{\vec{x}})$ is the identity map on $\vcplx$. By lemma \ref{functoriallemma} part 2 and Remark \ref{functlemmarmk}, for any $\pi\in S_n$, any composable morphisms $A,B$ in $\Zndiag$, and $w\in F^{\vec{x}}D^\pi_{\vec{x}}$, we have
\begin{align*}
(A\circ B)_{\#}(w)&=F^{\vec{x}\oplus\vec{y}}\big(C^\pi(A\circ B)\big)(w)\\
&=F^{\vec{x}\oplus\vec{y}}\big(C^\pi(A)\circ C^\pi(B)\big)(w)\\
&=\bigg(F^{\vec{x}\oplus\vec{y}}\big(C^\pi(A)\big)\circ F^{\vec{x}\oplus\vec{y}}\big(C^\pi(B)\big)\bigg)(w)\\
&=\big(A_{\#}\circ B_{\#}\big)(w)
\end{align*}
Thus the correspondence $\vec{x}\mapsto \mathcal{C}^*(D,\vec{x})$, and $A\mapsto A_{\#}$, is a functor. Post-composing with the cohomology functor gives the desired result.
\end{proof}

%

\section{Categorifying Determinants of Positive Integer Valued Matrices}
\label{categarbit}

Below, we describe our construction which works for all positive integer valued matrices, but has somewhat less interesting differentials and no combinatorial/topological interpretation. With the definition
$$\text{det}(m_{i,j})_{i,j=1}^n=\sum_{\pi\in S_n}(-1)^{\inv(\pi)}m_{1,\pi(1)}\dots m_{n,\pi(n)}$$
in mind, we apply the same process as in Section \ref{Sectionthree}, however skipping the detour through the world of colored cobordisms.

\begin{defn}
Let $M=(m_{i,j})_{i,j=1}^n$ be a matrix with entries in $\Z_+$. Let $\mathcal{A}=\{A_n\}_{n\in\Z_+}$ be a family of $\Z_2$-bialgebras $(A_n, \mu_n,\eta_n, \Delta_n,\varepsilon_n)$ where $\dim A_n=n$ for each $n\in\Z_+$. Construct a cochain complex $\mathcal{C}^*(M,\mathcal{A})$ via the following steps:
\begin{enumerate}[label=\arabic*.]
\item Start with the Hasse diagram of the (strong) Bruhat order on $S_n$. 
\item Replace each vertex $\pi\in S_n$ by the algebra
$A_\pi=A_{m_{1,\pi(1)}}\otimes \dots \otimes A_{m_{n,\pi(n)}}$. 
\item For each cover relation $\pi\lessdot\pi(i,j)$, replace the edge from $\pi$ to $\pi(i,j)$ by a map $f^{\pi,\pi(i,j)}:A_{\pi}\to A_{\pi(i,j)}$ by the rules:
\subitem \textbullet \ Apply the identity on all tensor factors $A_{m_{k,\pi(k)}}$ for which $k\notin\{i,j\}$.
\subitem \textbullet \ On the two tensor factors which do change we apply the maps $$A_{m_{i,\pi(i)}}\xrightarrow{\eta_{m_{i,\pi(j)}}\varepsilon_{{m_{i,\pi(i)}}}}A_{m_{i,\pi(j)}}, \ \ \ \ \ \ \ A_{m_{j,\pi(j)}}\xrightarrow{\eta_{m_{i,\pi(i)}}\varepsilon_{{m_{j,\pi(i)}}}}A_{m_{j,\pi(i)}}.$$
\item Define the cochain groups $\mathcal{C}^k(M,\mathcal{A})$ and a degree 1 codifferential $d^k:\mathcal{C}^k(M,\mathcal{A})\rightarrow \mathcal{C}^{k+1}(M,\mathcal{A})$ by 
$$\mathcal{C}^k(M,\mathcal{A})=\bigoplus_{\substack{\pi\in S_n\\ \text{inv}(\pi)=k}} A_{\pi},\ \ \ \  \ \ \ \ \ d^k(v)=\bigoplus_{\sigma\gtrdot \pi} f^{\pi,\sigma}(v)$$
for $v\in A_{\pi}$ with $\text{inv}(\pi)=k$. Let $\mathcal{H}^*(M,\mathcal{A})$ denote the cohomology groups of $\mathcal{C}^*(M,\mathcal{A})$. 
\end{enumerate}
\label{categarbitdefn}
\end{defn}
\begin{lemma} The complex $\mathcal{C}^*(M,\mathcal{A})$ constructed in Definition \ref{categarbitdefn} is indeed a cochain complex.
\end{lemma}
\begin{proof} By the definition of a bialgebra, we have $\epsilon_n\eta_n=1$ for each $n\in\Z_+$. The proof of this lemma is a slight modification of the proof of Lemma \ref{vandcomplexlemma} using this fact. \end{proof}
\begin{theorem} For any family $\mathcal{A}=\{A_n\}_{n\in\Z_+}$ of $\Z_2$-bialgebras with $\dim A_n=n$ for each $n\in\Z_+$, 
\[\sum_{k\geq 0}(-1)^k\dim\mathcal{H}^k(M,\mathcal{A})=\det(M).\]
\end{theorem}
\begin{proof} Repeat the proof of Theorem \ref{KirkMain} in the present context.\end{proof}


While this theorem appears more general than Theorem \ref{KirkMain}, the differentials used in the complex $\mathcal{C}^*(M,\mathcal{A})$ are somewhat trivial, and the cohomology groups are nearly the same size as the cochain groups. For this reason, one might desire a family $\{A_n\}_{n\in\Z_+}$ of algebras with more interesting maps $A_n\to A_m$. The (generalized) Vandermonde determinants were chosen as the focus of this paper because in the case that each row of our matrix consists of powers of the same number $x$, we can interpret the corresponding algebras as tensor powers of some fixed algebra $A_x$, and doing so allows for more interesting maps between algebras, utilizing the multiplication and comultiplication maps.

\section{Future Directions}
\label{lastsection}

This paper presents the first example of a cohomology theory categorifying a determinant. Hopefully the technique presented in this paper will lead to categorifications of other interesting determinants. We end with a list of questions related to the constructions presented in this paper. Questions 1-4 are the most general, relating to possible extensions or generalizations of the theory presented here. Questions 5-6 are related to Theorem \ref{genMain}.

\begin{enumerate}
\item Is there an upgrade to the construction presented in Section \ref{Sectionthree} using graded algebras in which we recover $V(\vec{x})$ as a polynomial by taking a graded Euler characteristic (as opposed to recovering an evaluation of a polynomial by taking an ordinary Euler characteristic)?
\item Is it possible to upgrade the construction presented in Section \ref{Sectionthree} so as to obtain a complex whose cohomology is a link invariant?
\item Applying a cofactor expansion along the last column in $V(\vec{x})=\det(x_i^j)_{i,j=1}^n$ yields
\begin{equation} V(\vec{x})=\sum_{k=1}^n(-1)^{n+k}x_k^n\ V\big(\vec{x}(i)\big)\label{cofactorexp}\end{equation}
where $V\big(\vec{x}(i)\big)=\det (x_i^j)_{(i,j)\in([n]\setminus \{k\})\times [n-1]}$ is the Vandermonde determinant in variables $\vec{x}(i)=(x_1,\dots,\widehat{x_i},\dots x_n)$ and is categorified by the complex $\mathcal{C}^*(T_{2,n-1},\vec{x}(i))$. Is there a resolution of $\tvcplx$ which categorifies the relation (\ref{cofactorexp})?
\item The first property one typically learns about the Vandermonde determinant is the product formula:
\begin{equation}
V(\vec{x})=x_1\dots x_n\prod_{i<j}(x_j-x_i).\label{eq10}
\end{equation}
One could consider categorifying the Vandermonde determinant using equation (\ref{eq10}) instead of (\ref{eq2}). This could be accomplished with a family of algebras $\{A_n\}_{n\in\N}$ with $\dim A_n=n$ by lifting the difference $x_j-x_i$ to the complex $A_{j,i}^*=\Cone(A_{x_j}\xrightarrow{\psi_{j,i}}A_{x_i})$
for some appropriate choice of maps $\psi_{j,i}$. For example, if the algebras $A_k$ are endowed with unit $\eta_k$ and counit $\varepsilon_k$, we could take $\psi_{j,i}=\eta_i\circ\varepsilon_j$. Then the complex
$
\wtilde{\mathcal{C}^*}(\vec{x})=A_{x_1}\otimes\dots\otimes A_{x_n}\bigotimes_{i<j}A_{j,i}^*,
$
has Euler characteristic $V(\vec{x})$. 
One may ask whether there are any interesting relations between $\wtilde{\mathcal{C}^*}(\vec{x})$ and the complex $\tvcplx$ or perhaps the complex constructed in Section \ref{categarbit} for the case of $V(\vec{x})$.
\item In the representation theory of $S_n$, quotients of generalized Vandermonde determinants by the Vandermonde determinant $V(\vec{x})$ in the same variables can be expressed as a sum of symmetric functions with coefficients related to characters of certain representations of $S_n$. Can this relation be formulated on the level of the complexes $\mathcal{C}^*(D,\vec{x})$? 
\item Which types of links admit diagrams with the property that the height of a smoothing determines the number of circles in the smoothing? 

\end{enumerate}

\bibliography{main.bib}

\end{document}